\documentclass[reqno,twoside]{amsart}
\usepackage{amsmath}
\usepackage{amsfonts}
\usepackage{amssymb}
\usepackage{xcolor}
\usepackage{enumerate}

\usepackage{a4wide,amsmath,amssymb,latexsym,amsthm}

\usepackage{mathrsfs,mathtools,epic,bm}
\usepackage{hyperref}
\hypersetup{colorlinks=true,linkcolor=blue,filecolor=mangeta,urlcolor= cyan}

\newtheorem{theorem}{Theorem}[section]
\newtheorem{proposition}[theorem]{Proposition}

\newtheorem{corollary}[theorem]{Corollary}
\theoremstyle{definition}
\newtheorem{definition}[theorem]{Definition}

\newtheorem{remark}[theorem]{Remark}

\numberwithin{equation}{section}

\def \dis {\displaystyle}

\def \R {\mathbb{R}}

\def \N {\mathbb{N}}
\def \H {\mathbb{H}}
\def \C {\mathcal{C}}

\def \V {\mathbb{ V}}

\def \H {\mathbb{ H}}

\def \C {{\mathcal C}}

\def \dx{\mathrm{d}x}

\def \dy {\mathrm{d}y}
\def \dq {\mathrm{d}x \, \mathrm{d}t}

\def\Omc{\mathbb{R}^N\setminus\Omega}
\def \ro {\rho}

\def \hvarphi \widehat{\varphi}

\def\RR{{\mathbb{R}}}

\def\bOm{\overline{\Om}}

\def\Om{\Omega}

\keywords{Fractional Laplacian, maximum principle.}

\begin{document}
	\title{On a weak maximum principle for a class of fractional diffusive equations}
	
	\author{Cyrille Kenne}
\address{Cyrille Kenne,  Laboratoire L.A.M.I.A., D\'{e}partement de Math\'{e}matiques et Informatique, Universit\'{e} des Antilles, Campus Fouillole, 97159 Pointe-\`{a}-Pitre,(FWI), Guadeloupe -
Laboratoire L3MA, UFR STE et IUT, Universit\'e des Antilles, Schoelcher, Martinique.}
\email{kenne853@gmail.com}
	
	\author{Gis\`{e}le Mophou}
\address{Gis\`{e}le Mophou, Laboratoire L.A.M.I.A., D\'{e}partement de Math\'{e}matiques et Informatique, Universit\'{e} des Antilles, Campus Fouillole, 97159 Pointe-\`{a}-Pitre,(FWI), Guadeloupe -
Laboratoire  MAINEGE, Universit\'e Ouaga 3S, 06 BP 10347 Ouagadougou 06, Burkina Faso.}
\email[Mophou]{gisele.mophou@univ-antilles.fr}

	\date{\today}
	
	\maketitle
	
\begin{abstract}
We consider two evolution equations involving space fractional Laplace operator of order $0<s<1$. We first establish some existence and uniqueness results for the considered evolution equations. Next, we give some comparison theorems and prove that, if the data of each equation are data bounded, then the solutions are also bounded.
\end{abstract}
	
\section{Introduction}
In this paper, we denote by $\Om\subset\R^N,\, N\in \N\setminus\{0\}$ an open bounded domain with a Lipschitz continuous boundary $\partial \Om$ and $\omega$ an open subset of $\Omega.$ For $T>0$, we set $Q:= \Omega\times (0,T)$ and $\Sigma:= (\R^N\setminus \Omega)\times (0,T)$. We  consider the space fractional diffusion model with non-homogeneous Robin boundary condition
\begin{equation}\label{modeln}
\left\{
\begin{array}{rllll}
\dis  \rho_t+(-\Delta)^s\rho &=&f \qquad &\mbox{in}& Q,\\
\mathcal{N}_s\rho+\rho&=&g  &\mbox{in}& \Sigma,\\
\rho(.,0)&=& \rho_0 &\mbox{in}& \Omega,
\end{array}
\right.
\end{equation}
and the space fractional diffusion model with homogeneous Dirichlet boundary condition
\begin{equation}\label{modeld}
\left\{
\begin{array}{rllll}
\dis  \rho_t+(-\Delta)^s\rho &=&f \qquad &\mbox{in}& Q,\\
\rho&=&0  &\mbox{in}& \Sigma,\\
\rho(.,0)&=& \rho_0 &\mbox{in}& \Omega,
\end{array}
\right.
\end{equation}
where $T>0$, the functions   $\rho_0\in L^2(\Omega)$, $f\in L^2(Q)$ and $g\in L^2(\Sigma)$. $(-\Delta)^s$ ($0<s<1$) denotes the fractional Laplace operator given formally for a suitable function $\psi$ by:
\begin{align*}
(-\Delta)^s\psi=C_{N,s}\,\mbox{P.V.}\;\int_{\RR^N}\frac{\psi(x)-\psi(y)}{|x-y|^{N+2s}}\;dy,
\end{align*}
where P.V denotes the principal value and $C_{N,s}$ is a normalization constant depending only on $N$ and $s$. We refer to Section \ref{prelim} for more details.
We aim to establish the weak maximum principle results for the systems \eqref{modeln}-\eqref{modeld}.\par
Recently, the fractional calculus has become increasingly considered by researchers. It has been used for many application, for instance, modeling of  chemical and physical process, in engineering and finance (see e.g \cite{chechkin2005, freed2002, hilfer2000, kilbas2006, metzler2000, milchev2009, podlubny1998} for instance and the references therein).
With this emergence of fractional calculus, the maximum principles results have gained much attention.
Theses techniques are used to obtain more information about a solution even if we do not have its explicit expression.
They are also of a good use to prove existence and uniqueness results for partial differential equations.
They can also be used  in control theory to prove the uniqueness of an optimal control, specially for bilinear systems.\par
The maximum principles for fractional order has been widely concerned by authors. In \cite{luchko2009},
Luchko formulated and established a maximum principle result for a time-fractional diffusion equation.
It was done by using an extremum principle established earlier for the Caputo-Dzherbashyan fractional derivative.
This result was used later in \cite{luchk2010}, to prove well posedness results of a initial-boundary-value problem
for the multi-dimensional time-fractional diffusion equation. Al-Refai \textit{et al}. \cite{al2014}
established maximum principles for a fractional diffusion equation involving the Riemann-Liouville fractional derivative.
They employed these results to show an uniqueness result for solutions to a initial-boundary-value problem for
the nonlinear fractional diffusion equation under some standard assumptions.\par

We point out that in the above papers, there is no consideration of the fractional Laplace operator. But, this latter is a non local operator and has found applications in many areas see \cite{caffarelli2010}, for instance. For more topics on fractional Laplace and its applications, we refer to \cite{ barrios2012,  brandle2013, cabre2014,caffarelli2008, silvestre2007, viet2019} and the references therein.
Concerning maximum principle involving fractional  operator, Wang \textit{et al.} \cite{wang2020} established maximum principles results for Hadamard fractional differential equation with fractional Laplace operator. They applied this result to a linear and nonlinear Hadamard fractional equations in order to prove the uniqueness and continuous dependence of the solution to the initial-boundary-value problem.

More recently, Barrios \textit{et al.} \cite{musina2019}, studied comparison results for a class of elliptic and
parabolic problems involving the fractional Laplace operator. They proved a generalization of a Hopf's lemma in a non-local form for elliptic and parabolic problems with mixed conditions. Trong \textit{et al.} \cite{trong2021}, proved a strong maximum principle result for the spectral Dirichlet Laplacian. The result was obtained by using the  Poincar\'e's inequality in a fractional version and by means of a certain truncated function.

In this article, we study two evolution equations involving space fractional Laplace operator. After some well posedness results of our systems, we establish weak maximum principles.
The rest of the work is organized as follows.  We introduce some Sobolev spaces and recall some known results on Fractional Laplace operator in Section \ref{prelim}. In Section \ref{existe}, the existence and uniqueness of solutions to systems \eqref{modeln}-\eqref{modeld} is established. We derive maximum principles results in Section \ref{maximum}. In Section \ref{conclusion}, a conclusion is given.
\section{Preliminaries}
\label{prelim}
We start this section by recalling the definition of the fractional Laplace operator. Given $0<s<1$,  we let the space
\begin{align*}
\mathcal L_s^{1}(\RR^N)\coloneqq \left\{u:\RR^N\to\RR\;\mbox{ measurable and}\; \int_{\RR^N}\frac{|u(x)|}{(1+|x|)^{N+2s}}\;\dx<\infty\right\}.
\end{align*}
Let $w\in \mathcal L_s^{1}(\RR^N)$ and $\varepsilon>0$, we set
\begin{align*}
(-\Delta)_\varepsilon^s w(x)\coloneqq C_{N,s}\int_{\{y\in\RR^N:\;|x-y|>\varepsilon\}}\frac{w(x)- w(y)}{|x-y|^{N+2s}}\;\dy,\;\;x\in\RR^N,
\end{align*}
where $C_{N,s}$, normalization constant is given by
\begin{align}\label{CNs}
C_{N,s}\coloneqq \frac{s2^{2s}\Gamma\left(\frac{2s+N}{2}\right)}{\pi^{\frac{N}{2}}\Gamma(1-s)}.
\end{align}
We define the fractional Laplace operator $(-\Delta)^sw$ as the following singular integral:
\begin{equation}\label{fl_def}
(-\Delta)^sw(x)\coloneqq C_{N,s}\,\mbox{P.V.}\int_{\RR^N}\frac{w(x)-w(y)}{|x-y|^{N+2s}}\;\dy =
\lim_{\varepsilon\downarrow 0}(-\Delta)_\varepsilon^s w(x),\;\;x\in\RR^N,
\end{equation}
provided that the limit exists for a.e. $x\in \R^N$. We refer to \cite{NPV} and the references therein for the class of functions for which the  limit in \eqref{fl_def} exists for a.e. $x\in\RR^N$.\par

Let $\Omega\subset\RR^N$ ($N\ge 1$) be an arbitrary open set and $0<s<1$. We define the fractional order Sobolev space
	$$
	H^s(\Omega):=\left\{\rho\in L^2(\Omega):\; \int_{\Omega}\int_{\Omega}\frac{|\rho(x)-\rho(y)|^2}{|x-y|^{N+2s}}\;dxdy<\infty\right\}
	$$
	and we endow it with the norm given by
	$$
	\|\rho\|_{H^s(\Omega)}=\left(\int_{\Omega}|\rho|^2\;dx+\int_{\Omega}\int_{\Omega}\frac{|\rho(x)-\rho(y)|^2}{|x-y|^{N+2s}}\;dxdy\right)^{1/2}.
	$$
We introduce for any  $\rho\in H^{s}(\RR^N)$ the {\em nonlocal normal derivative $\mathcal N_s$} defined by
	\begin{equation}\label{NonlocalDeri}
	\mathcal N_{s}\rho(x)\coloneqq C_{N,s}\int_{\Omega}\frac{\rho(x)-\rho(y)}{|x-y|^{N+2s}}\; \dy,~~~~x\in\RR^N\setminus\bOm,
	\end{equation}
	where $C_{N,s}$ is the normalization constant given in \eqref{CNs}.  We have that $\mathcal N_{s}$ maps $\rho\in H^{s}(\RR^N)$ continuously into $L^2_{loc}(\RR^N\setminus\bOm)$ (see \cite[Lemma 2.5, page 1604]{claus2020realization}).\\
We have the integration by parts formula which can be found in \cite{Dipierro2017, War-ACE} for smooth functions. Let $\rho\in H^{s}(\RR^N)$ be such that $(-\Delta)^s \rho\in L^2(\Omega)$ and $\mathcal N_s\rho\in L^2(\Omc)$. Then for every $\psi\in H^{s}(\RR^N)$, the following equality
	\begin{align}\label{Int-Part}
	\frac{C_{N,s}}{2}\int\int_{\RR^{2N}\setminus(\Omc)^2}&
	\frac{(\rho(x)-\rho(y))(\psi(x)-\psi(y))}{|x-y|^{N+2s}}\;\dx\, \dy\notag\\
	=&\int_{\Omega}\psi(-\Delta)^s\rho\;\dx+\int_{\Omc}\psi\mathcal N_s \rho\;\dx,
	\end{align}
	holds.\par

	Hence, if
	$\rho,\,\psi\in H^{s}(\RR^N)$ with  $(-\Delta)^s \rho,\, (-\Delta)^s \psi\in L^2(\Omega)$ and $\mathcal N_s\rho,\, \mathcal N_s\psi\in L^2(\Omc)$. Then  the following identity holds,
	\begin{equation}\label{Int-Part2}
	\begin{array}{llll}
	\dis \int_{\Omega}\psi(-\Delta)^s\rho\;\dx&=&\dis \int_{\Omega}\rho(-\Delta)^s\psi\;\dx
	+\dis \int_{\Omc}\rho\mathcal N_s \psi\;\dx\\
	&-&\dis\int_{\Omc}\psi\mathcal N_s \rho\;\dx.
	\end{array}
	\end{equation}
	It is clear that if $\rho=0$ in $\Omc$ or $\psi=0$ in $\Omc$, then
	$$
	\int\int_{\RR^{2N}\setminus(\RR^N\setminus\Omega)^2}
	\frac{(\rho(x)-\rho(y))(\psi(x)-\psi(y))}{|x-y|^{N+2s}}\dx\,\dy=
	\int_{\R^N}\int_{\R^N}\frac{(\rho(x)-\rho(y))(\psi(x)-\psi(y))}{|x-y|^{N+2s}}\dx\,\dy,
	$$
	and the integration by part formula \eqref{Int-Part} becomes
	\begin{equation}\label{Int-Part0}
	\frac{C_{N,s}}{2}\int\int_{\RR^{2N}\setminus(\Omc)^2}
	\frac{(\rho(x)-\rho(y))(\psi(x)-\psi(y))}{|x-y|^{N+2s}}\;\dx\, \dy
	=\int_{\Omega}\psi(-\Delta)^s\rho\;\dx\quad \forall \rho,\psi\in H^{s}(\RR^N)
	\end{equation}

We set
	$$
	\H_0^{s}(\Omega):=\Big\{\rho\in H^s(\R^N):\;\rho=0\;\hbox{ in }\;\RR^N\setminus\Omega\Big\}.
	$$
	Then,  endowed with the norm
	\begin{equation}\label{norm0}
	\|\rho\|_{\H_0^{s}(\Omega)}:=\left(\frac{C_{N,s}}{2}\int_{\R^N}\int_{\R^N}
	\frac{(\rho(x)-\rho(y))^2}{|x-y|^{N+2s}}\;\dx\, \dy\right)^{1/2},
	\end{equation}
	$\H_0^{s}(\Omega)$  is a Hilbert space (see e.g. \cite[Lemma 7]{servadei}). Let  $\H^{-s}(\Omega):=(\H_0^s(\Omega))^\star$ be  the dual space of $\H_0^s(\Omega)$ with respect to the pivot space $L^2(\Omega)$. Then we have the following continuous embedding (see e.g. \cite{ATW}):
	\begin{equation}\label{injection1}
	\H_0^{s}(\Omega)\hookrightarrow L^2(\Omega)\hookrightarrow \H^{-s}(\Omega).
	\end{equation}	
	From now on,  for any $\rho,\psi\in \H_0^{s}(\Omega)$, we set
	\begin{equation}\label{defF}
	\mathcal{F}(\rho,\psi):= \frac{C_{N,s}}{2} \int_{\R^N}\int_{\R^N}
	\frac{(\rho(x)-\rho(y))(\psi(x)-\psi(y))}{|x-y|^{N+2s}}\;\dx\, \dy.
	\end{equation}
	Therefore,  the norm on $\H_0^{s}(\Omega)$ given by \eqref{norm0} becomes
	$\|\rho\|_{\H_0^{s}(\Omega)}=\left(\mathcal{F}(\rho,\rho)\right)^{1/2}$.\par
Let  $(-\Delta)_D^s$ be the operator defined on $L^2(\Omega)$  by
	\begin{equation}\label{FDL}
	D((-\Delta)_D^s):=\{\rho\in \H_0^s(\Omega):\;(-\Delta)^s\rho\in L^2(\Omega)\},\; (-\Delta)_D^s\rho:=(-\Delta)^s\rho\;\hbox{ in }\Omega.
	\end{equation}
	Then,  $(-\Delta)_D^s$ is the realization in $L^2(\Omega)$ of $(-\Delta)^s$ with the zero Dirichlet exterior condition.\par
The following result is well-known (see e.g.  \cite{claus2020realization,GW-CPDE}).
	\begin{proposition}\label{Prop-22}
		Let $(-\Delta)_D^s$ be the operator defined in \eqref{FDL}.
		Then, $(-\Delta)_D^s$ can be also viewed as a bounded operator from $H_0^s(\Omega)$ into $H^{-s}(\Omega)$ given by
		\begin{align}\label{FDLB}
		\langle (-\Delta)_D^s\rho,\varphi\rangle_{H^{-s}(\Omega),H_0^s(\Omega)}=\mathcal F(\rho,\varphi),\;\;\; \rho,\varphi\in \H_0^s(\Omega).
		\end{align}
\end{proposition}
\begin{remark}\label{rmk0}
	$ $
	Note that  if $N>2s,$ then   there is a positive constant $C_0=C(N,s)>0$ such that for any $\rho \in \H_0^s(\Omega)$,
	\begin{equation}\label{normeq}
	\|\rho\|_{L^r(\Omega)}\leq C_0\|\rho\|_{\H_0^s(\Omega)},
	\end{equation}
	where $r=\dis \frac{2N}{N-2s}$ (see e.g. \cite[Lemma 6 a)]{servadei}). Therefore,
	\begin{equation}\label{normeq1}
	\|\rho\|_{L^2(\Omega)}\leq C_0\|\rho\|_{\H_0^s(\Omega)},
	\end{equation}
	because $\Omega$ is an open and bounded subset of $\R^N$ and $r>2$. Moreover, the following holds true,
	\begin{equation}\label{normeq1bis}
	\|\rho\|_{L^2((0,T);L^2(\Omega))}\leq C_0\|\rho\|_{L^2((0,T);\H_0^s(\Omega))}.
	\end{equation}
\end{remark}

Next, let $0<s<1$, we consider the space
\begin{equation}\label{defHS}
\dis \H^{s}_{\Omega}\coloneqq \left\{\rho:\R^N\to \R \hbox{ measurable  such that   } \|\rho\|_{\H^{s}_{\Omega}} <\infty\right\},
\end{equation}
where
\begin{equation}\label{defmormHS}
	\dis \|\rho\|_{\H^{s}_{\Omega}}\coloneqq\dis \left(\|\rho\|^2_{L^2(\R^N)}+\frac{C_{N,s}}{2} \int\int_{\RR^{2N}\setminus(\Omc)^2}
	\frac{(\rho(x)-\rho(y))(\psi(x)-\psi(y))}{|x-y|^{N+2s}}\;\dx\, \dy\right)^{\frac 12}.
	\end{equation}
Then proceeding  exactly as for the proof of \cite[Proposition 3.1, page 385]{Dipierro2017}, we have that  $ \H^{s}_{\Omega}$  endowed with the norm \eqref{defmormHS}   is a Hilbert space.

\begin{remark}For any $\rho$ and $\psi$ in $\H^{s}_{\Omega}$, we set
	\begin{equation}\label{defFR}
	\mathcal{F}_R(\rho,\psi):= \frac{C_{N,s}}{2} \int\int_{\RR^{2N}\setminus(\Omc)^2}
	\frac{(\rho(x)-\rho(y))(\psi(x)-\psi(y))}{|x-y|^{N+2s}}\;\dx\, \dy.
	\end{equation}
	Then the norm on $\H^{s}_{\Omega}$ becomes
\begin{equation}\label{defmormHSbis}
		\dis \|\rho\|_{\H^{s}_{\Omega}}\coloneqq\dis \left(\|\rho\|^2_{L^2(\R^N)}+\mathcal{F}_R(\rho,\rho)\right)^{\frac 12}.
		\end{equation}
	Note also that from \eqref{defmormHS}, we have
	\begin{equation}\label{remnormHS}
	\begin{array}{lllll}
	\|\rho\|_{L^2(\Sigma)}&\leq& \|\rho\|_{\H^{s}_{\Omega}},\\
	\|\rho\|_{L^2(\Omega)}&\leq& \|\rho\|_{\H^{s}_{\Omega}}.
	\end{array}
	\end{equation}
\end{remark}
Let $(\H^s_\Omega)^\star$  the dual of $\H^s_\Omega.$ Then
	\begin{equation}\label{injectionHS}
	\H^s_\Omega\hookrightarrow L^2(\Omega)\hookrightarrow(\H^s_\Omega)^\star.
	\end{equation}

We denote by $(-\Delta)^s_R$ on $L^2(\Omega)$, the operator defined as follows:
	\begin{equation}
	\label{domR}D((-\Delta)^s_R)=\left\{\rho\in \H^{s}_{\Omega}\,:\, (-\Delta)^s\rho\in L^2(\Omega), \, \mathcal{N}_s\rho\in L^2(\R^N\setminus \Omega)\right\};\quad (-\Delta)^s_R\rho=(-\Delta)^s \rho \hbox{ in } \Omega.
	\end{equation}
	\begin{proposition}\label{proprobin0}
		Let $(-\Delta)^s_R$ be the operator defined in \eqref{domR}. Then  $(-\Delta)^s_R$ can be viewed as a bounded operator from $\H^{s}_{\Omega}$ into $(\H^{s}_{\Omega})^\star$ given by
		$$\langle (-\Delta)^s_R\rho,\phi\rangle_{(\H^s_\Omega)^\star,\H^s_\Omega}=\mathcal{F}_R(\rho,\phi)-
		\dis \int_{\R^N\setminus \Omega} \mathcal{N}_s\rho \phi dx \quad \forall \rho,\phi\in \H^s_\Omega$$
\end{proposition}
\begin{proof}
	From  \eqref{domR} and the integration by parts formula \eqref{Int-Part} we have
	$$
	\int_{\Omega}\phi\, (-\Delta)^s_R\rho\;\dx=\int_{\Omega}\phi(-\Delta)^s\rho\;\dx=\mathcal{F}_R(\rho,\phi)-\int_{\Omc}\phi\mathcal N_s \rho\;\dx,\quad \forall \phi,\rho\in \H^s_\Omega.
	$$
	Then using the definition of the norm of $\H^s_\Omega$ given by \eqref{defmormHS} and \eqref{remnormHS}, we have that,
	$$\left|\mathcal{F}_R(\rho,\phi)-
	\dis \int_{\R^N\setminus \Omega} \mathcal{N}_s\rho \phi dx \right|\leq
	\left(\mathcal{F}_R(\rho,\rho)+\|\mathcal{N}_s\rho\|^2_{L^2(\R^N\setminus \Omega)}\right)^{1/2}\|\phi\|_{\H^s_\Omega}$$
	Therefore  for any $\phi\in \H^s_\Omega,$
	$$
	\left|\int_{\Omega}\phi\, (-\Delta)^s_R\rho\;\dx\right|\leq
	\left(\mathcal{F}_R(\rho,\rho)+\|\mathcal{N}_s\rho\|^2_{L^2(\R^N\setminus \Omega)}\right)^{1/2}\|\phi\|_{\H^s_\Omega}$$
	and we deduce that
	$$\|(-\Delta)^s_R\rho\|_{(\H^s_\Omega)^\star}\leq \left(\mathcal{F}_R(\rho,\rho)+\|\mathcal{N}_s\rho\|^2_{L^2(\R^N\setminus \Omega)}\right)^{1/2}.$$
\end{proof}

	Let $\mathbb{X}^\star$ be the dual of $\mathbb{X}$. Set
	\begin{equation}\label{defW0T}
	W(0,T;\mathbb{X}):= \left\{\zeta \in L^2(0,T;\mathbb{X}): \zeta_{t} \in L^2\left((0,T);\mathbb{X}^\prime\right)\right\}.
	\end{equation}
	Then $W(0,T;\mathbb{X})$ endowed with the  norm given by
	\begin{equation}\label{normW0T}
	\|\psi\|^2_{W(0,T;\mathbb{X})}=\|\psi\|^2_{L^2(0,T;\mathbb{X})}+\|\psi_t\|^2_{
		L^2\left(0,T;\mathbb{X}^\star\right)},\,\forall \psi \in W(0,T;\mathbb{X}),
	\end{equation}
	is a Hilbert space.  Moreover,  if $\mathbb{Y}$ is a Hilbert space that can be identified with its dual $\mathbb{Y}^\star$  and we have the continuous embeddings $\mathbb{X}\hookrightarrow \mathbb{Y}=\mathbb Y^\star\hookrightarrow \mathbb{X}^\star,$
	then using \cite[Theorem 1.1, page 102]{lions1971},  we have the continuous embedding
	\begin{equation}\label{contWTA}
	W(0,T;\mathbb{X})\hookrightarrow C([0,T];\mathbb{Y}).
	\end{equation}
\section{Existence results}
\label{existe}
In this section, we are interested in the existence and uniqueness of solutions to system  \eqref{modeln}.
To this end, we start by studying  the existence and uniqueness of and auxiliary model. So,
we consider the auxiliary system:
\begin{equation}\label{modelexist1aux}
\left\{
\begin{array}{rllll}
\dis  z_t+(-\Delta)^sz+ z &=&\zeta \qquad &\mbox{in}& Q,\\
\mathcal{N}_sz+z&=& \eta &\mbox{in}& \Sigma,\\
z(\cdot,0)&=& \rho_0 &\mbox{in}& \Omega,
\end{array}
\right.
\end{equation}
where
\begin{equation}\label{inter1}
\zeta=e^{-t}f \hbox{ and }\eta=e^{-t}g.
\end{equation}
\begin{definition}\label{weaksolutionaux}
Let  $\zeta\in L^2((0,T);(\H^s_\Omega)^\star),$ $\eta\in L^2(\Sigma)$  and $\rho_{0}\in L^2(\Omega)$. Let  $\mathcal{F}_R(\cdot,\cdot)$ be defined by \eqref{defFR}.  A function
	$z\in  W(0,T;\H^s_\Omega)$ is said to be a weak solution to \eqref{modelexist1aux}, if the following equality holds:
	\begin{equation}\label{Eq-Def31aux}
	\begin{array}{lll}
\dis -\int_0^T\langle \phi_t,z \rangle_{(\H^s_\Omega)^\star,\H^s_\Omega}\, dt +
	\dis\int_0^T \mathcal{F}_R(z,\phi)dt+\int_Q  \, z\, \phi dx\,dt+\int_\Sigma z\, \phi \dq\\
	=\dis \int_0^T \langle \zeta,\, \phi \rangle_{(\H^s_\Omega)^\star,\H^s_\Omega} dt+
	\int_\Sigma \eta\, \phi \dq+\int_\Omega \rho_0\,\phi(\cdot,0) \dx,\quad \forall \phi \in H(Q),
	\end{array}
	\end{equation}
	where
	$$ H(Q):=\left\{\xi\in W(0,T;\H^s_\Omega) \hbox{ and } \xi(\cdot, T)=0 \hbox{ a.e.in } \Omega\right\}.$$
\end{definition}
\begin{remark}\label{existvarphi}
	Note also that if $ \varphi \in H(Q)$, then $\varphi \in W(0,T;\H^s_\Omega)$. Therefore, we have from \eqref{contWTA} that  $\varphi(0)$ and $\varphi(T)$ exist and belong to $L^2(\Omega)$.	
\end{remark}
\begin{theorem}\label{exist1aux}
Let $\zeta\in L^2((0,T);(\H^s_\Omega)^\star),$ $\eta\in L^2(\Sigma)$  and $\rho_{0}\in L^2(\Omega)$.  Then, there exists a unique weak solution $z\in W(0,T;\H^s_\Omega)$ to \eqref{modelexist1aux} in the sense of Definition~\ref{weaksolutionaux}.
	In addition, the following estimates hold:
	\begin{equation}\label{estimation1aaux}
	\dis \sup_{\tau\in [0,T]}\|z(\tau)\|_{L^2(\Omega)}\leq  \dis
	\left( \|\zeta\|^2_{L^2((0,T);(\H^s_\Omega)^\star)}+\|\eta\|^2_{L^2(\mathcal O\times(0,T))}+\dis  \|\rho_0\|^2_{L^2(\Omega)}\right)^{1/2},
	\end{equation}	
	\begin{equation}\label{estimation1baux}
	\dis \|z\|_{L^2((0,T);\H^s_\Omega)}\leq
	\dis \left(\|\zeta\|_{L^2((0,T);(\H^s_\Omega)^\star)}+\|\eta\|_{L^2(\Sigma)}+\dis  \|\rho_0\|_{L^2(\Omega)}\right).
	\end{equation}
	and there exists a constant $C>0$ such that
	\begin{equation}\label{estimation1bauxt}
	\dis \|z\|_{W(0,T;\H^s_\Omega)}\leq C\dis \left(\|\zeta\|_{L^2((0,T);(\H^s_\Omega)^\star)}+\|\eta\|_{L^2(\Sigma)}+\dis  \|\rho_0\|_{L^2(\Omega)}\right).
	\end{equation}
\end{theorem}
\begin{proof}The proof is done in three steps.\par	
	\noindent \textbf{Step 1.} We proceed as in \cite[Page 37]{lions2013} to prove the existence of solutions. To this end, we recall that the norm on $L^2((0,T);\H_\Omega^s)$ is given by
	$$\|\varphi\|^2_{L^2((0,T);\H_\Omega^s)}=\int_0^T \|\varphi(\cdot,t)\|^2_{\H^s_\Omega} dt=
	\int_0^T\left(\|\varphi(\cdot,t)\|^2_{L^2(\R^N)}+
	\mathcal{F}_R(\varphi(\cdot,t),\varphi(\cdot,t))\right)dt$$
	and we consider the norm on  $H(Q)$ given by
	$$
	\|\rho\|^2_{H(Q)}:=\|\rho\|^2_{L^2((0,T);\H^s_\Omega)}+\|\rho(\cdot,0)\|^2_{L^2\Omega)},\, \forall \rho \in H(Q).
	$$
	It is clear that for any $\rho\in H(Q),$
	$$\|\rho\|_{L^2((0,T);\H_\Omega^s)}\leq \|\rho\|_{H(Q)}.$$
	This shows that embedding $H(Q)\hookrightarrow L^2((0,T);\H_\Omega^s)$ is continuous .\par	
	Now, let $\varphi \in H(Q)$ and consider the bilinear form $\mathcal{E}(\cdot,\cdot)$ defined on $L^2((0,T);\H_\Omega^s)\times H(Q)$ by
\begin{equation}\label{defCalE1}
	\begin{array}{lll}
	\mathcal{E}(z,\varphi)&:=&\dis -\int_0^T\langle \varphi_t,z \rangle_{(\H^s_\Omega)^\star,\H^s_\Omega}\, dt
	+\dis\int_0^T \mathcal{F}_R(z,\varphi)dt+\int_\Sigma z\varphi \dq+\int_Q  \, z\, \varphi dx\,dt,
	\end{array}
	\end{equation}
	where  $\mathcal{F}_R(\cdot,\cdot)$ is  defined as in \eqref{defFR}.
	Using Cauchy-Schwarz's inequality, the  continuity of the bilinear form $\mathcal{F}_R$ and  \eqref{remnormHS},
	it follows that there exists a constant $C=C(\varphi,N,s)>0$ such that
$$\begin{array}{rlll}
	\dis |\mathcal{E}(z,\varphi)|&\leq&\dis  \|z\|_{L^2((0,T);\H_\Omega^s)}\|\varphi_{t}\|_{L^2((0,T);(\H^s_\Omega)^\star)}+\dis
	\left(\int_0^T \mathcal{F}_R(z,z)dt\right)^{1/2}\left(\int_0^T \mathcal{F}_R(\varphi,\varphi)dt\right)^{1/2}\\
	&+&\|z\|_{L^2(\Sigma)}\|\varphi\|_{L^2(\Sigma)}+\|z\|_{L^2(Q)}\|\varphi\|_{L^2(Q)}\\
	&\leq& \|z\|_{L^2((0,T);\H_\Omega^s)}\Big(\|\varphi_{t}\|_{L^2((0,T);(\H^s_\Omega)^\star)} +3\|\varphi\|_{L^2((0,T);\H_\Omega^s)}\Big) .
	\end{array}
	$$
	We have shown that there is a constant $C=C(\varphi,N,s)>0$  such that
	$$|\mathcal{E}(z,\varphi)|\leq C\|z\|_{L^2((0,T);\H_\Omega^s)}. $$
	Consequently, for every fixed $\varphi\in H(Q),$
	the functional  $z\mapsto \mathcal{E}(z,\varphi)$ is continuous on $L^2((0,T);\H_\Omega^s).$\par	
	Next, using a simple integration by parts, it follows that for every  $\varphi\in H(Q)$,
	$$\begin{array}{rlll}
	\mathcal{E}(\varphi,\varphi)&=&\dis -\int_0^T\langle \varphi_t,\varphi \rangle_{(\H^s_\Omega)^\star,\H^s_\Omega}\, dt
	+\dis\int_0^T \mathcal{F}_R(\varphi,\varphi)dt+\int_\Sigma \varphi^2 \dq+r\int_Q  \varphi^2 dx\,dt\dq
	\\
	&\geq&\dis \frac 12  \|\varphi(0)\|^2_{L^2(\Omega)}+\dis\int_0^T \mathcal{F}_R(\varphi,\varphi)dt+\int_\Sigma |\varphi|^2 \dq+\|\varphi\|^2_{L^2(Q)}\\
	&\geq &\dis \frac 12 \|\varphi(0)\|^2_{L^2(\Omega)}+ \dis\|\varphi\|^2_{L^2((0,T);\H^s_\Omega)}\\
	&\geq &
	\dis \dis\frac{1}{2} \|\varphi\|^2_{H(Q)}.
	\end{array}
	$$	
	Finally, we define the functional $L(\cdot):H(Q)\to \R$  by
	$$
	L(\varphi):=\dis \int_0^T \langle \zeta,\, \varphi \rangle_{(\H^s_\Omega)^\star,\H^s_\Omega} dt+\int_\Sigma \eta\, \varphi \dq+\int_\Omega \rho_0\,\varphi(0) \dx \quad \forall \varphi \in H(Q).
	$$
	Using Cauchy-Schwarz's inequality and \eqref{remnormHS}, we get that
	$$\begin{array}{rll}
	|L(\varphi)|&\leq&\|\varphi(0)\|_{L^2(\Omega)}\|\rho_0\|_{L^2(\Omega)}+\|\zeta\|_{L^2((0,T);(\H^s_\Omega)^\star)}
	\|\varphi\|_{L^2((0,T);\H^s_\Omega)}+\|\eta\|_{L^2(\Sigma)}\|\varphi\|_{L^2(\Sigma)}\\
	&\leq &\left(\|\rho_0\|_{L^2(\Omega)}+\|\zeta\|_{L^2((0,T);(\H^s_\Omega)^\star)}+\|\eta\|_{L^2(\Sigma)}\right)\|\varphi\|_{H(Q)}.
	\end{array}
	$$
	Therefore,  the functional $L(\cdot)$ is continuous on $H(Q)$. Thus, there is $z\in L^2((0,T);\H_\Omega^s)$ such that
	$$
	\mathcal{E}(z,\varphi)= L(\varphi),\quad \forall \varphi \in H(Q).
	$$
	Hence, according to Definition \ref{weaksolutionaux}, the system \eqref{modelexist1aux} has a solution $z\in L^2((0,T);\H_\Omega^s)$.\\
	\noindent \textbf{Step 2.} We show that $z_t\in L^2((0,T);(\H_\Omega^s)^\star)$\par
		Since $z\in L^2((0,T);\H_\Omega^s)$,  it follows from Proposition \ref{proprobin0} that $(-\Delta)_R^sz(\cdot,t)\in (\H_\Omega^s)^\star$ and thus
		$z_t(\cdot,t)=-(-\Delta)^s_Rz(\cdot,t) -z(\cdot,t)+ \zeta(\cdot,t)\in (\H_\Omega^s)^\star.$
		If we multiply the first equation in \eqref{modelexist1aux} by $\phi\in L^2((0,T);\H_\Omega^s)$,  and use the integration by parts formula \eqref{Int-Part},  we obtain
		$$\begin{array}{rlll}
		\dis \int_{\R^N\setminus\Omega}\eta(t)\, \phi(t) \,\dx+\dis
		\left\langle \zeta(t),\phi(t)\right\rangle_{(\H_\Omega^s)^\star,\H_\Omega^s}&=&	
		\dis  \left\langle z_t(t),\phi(t)\right\rangle_{(\H_\Omega^s)^\star,\H_\Omega^s} +\dis  \mathcal{F}_R(z(t),\phi(t))\\
		&+&\dis \int_{\Omega}z(t)\, \phi(t) \,\dx+\int_{\R^N\setminus\Omega}z(t)\, \phi(t) \,\dx
		\end{array}
		$$
		This implies that
		$$\begin{array}{rlll}
		\dis \left|  \left\langle z_t(t),\phi(t)\right\rangle_{(\H_\Omega^s)^\star,\H_\Omega^s}\right|&\leq&
		\|\eta(t)\|_{L^2(\R^N\setminus\Omega)}\|\phi(t)\|_{L^2(\R^N\setminus\Omega)}+
		\|\zeta(t)\|_{(\H_\Omega^s)^\star}\|\phi(t)\|_{\H_\Omega^s}\\
		&+& \left(\mathcal{F}_R(z(t),z(t))\right)^{1/2}\left(\mathcal{F}_R(\phi(t),\phi(t))\right)^{1/2}+
		\|z(t)\|_{L^2(\R^N\setminus\Omega)}\|\phi(t)\|_{L^2(\R^N\setminus\Omega)}\\
		&+&\|z(t)\|_{L^2(\Omega)}\|\phi(t)\|_{L^2(\Omega)}\\
		&\leq&2\|\phi(t)\|_{\H_\Omega^s}\left(
		\|\eta(t)\|^2_{L^2(\R^N\setminus\Omega)}+\|\zeta(t)\|^2_{(\H_\Omega^s)^\star}+
		3\|z(t)\|^2_{\H_\Omega^s}\right)^{1/2}\\
		&\leq&C\|\phi(t)\|_{\H_\Omega^s}\left(
		\|\eta(t)\|_{L^2(\R^N\setminus\Omega)}+\|\zeta(t)\|_{(\H_\Omega^s)^\star}+
		\|z(t)\|_{\H_\Omega^s}\right).
		\end{array}
		$$
		This means that
		\begin{equation}\label{sam1}
		\dis \left|  \left\langle z_t(t),\phi(t)\right\rangle_{(\H_\Omega^s)^\star,\H_\Omega^s}\right|\leq C\|\phi(t)\|_{\H_\Omega^s}\left(
		\|\eta(t)\|_{L^2(\R^N\setminus\Omega)}+\|\zeta(t)\|_{(\H_\Omega^s)^\star}+
		\|z(t)\|_{\H_\Omega^s}\right),
		\end{equation}
		where $C>0.$\\
Integrating \eqref{sam1} over $(0,T)$, we get
		\begin{equation}\label{sam2}
		\dis \int_0^T \left| \left\langle z_t(t),\phi(t)\right\rangle_{(\H_\Omega^s)^\star,\H_\Omega^s}\right| dt
		\leq C\|\phi\|_{L^2((0,T);\H_\Omega^s)}\left(
		\|\eta\|_{L^2(\Sigma}+\|\zeta\|_{L^2((0,T);(\H_\Omega^s)^\star)}+
		\|z\|_{L^2((0,T);\H_\Omega^s)}\right)\end{equation}
		Using \eqref{estimation1baux} we get from \eqref{sam2} that
		\begin{equation} \label{estimationint3}
		\|z_t\|_{L^2((0,T);(\H_\Omega^s)^\star)}\leq C\left(
		\|\eta\|_{L^2(\Sigma)}+\|\zeta(t)\|_{L^2((0,T);(\H_\Omega^s)^\star)}+
		\|\rho_0\|_{L^2(\Omega)}\right).
		\end{equation}
		Thus, $z_t\in L^2((0,T);(\H_\Omega^s)^\star)$ and we have shown that $z\in W(0,T;\H_\Omega^s)$.\par
	\noindent \textbf{Step 3.} We show the estimates \eqref{estimation1aaux}-\eqref{estimation1bauxt}.\par
	Multiplying the first equation in \eqref{modelexist1aux}  by $z\in L^2((0,T);\H^s_\Omega)$  and integrating by parts, we obtain using  \eqref{Int-Part} that,
	$$\begin{array}{rll}
	\dis  \frac 12\frac{d}{dt}\|z(t)\|^2_{L^2(\Omega)}+\dis \mathcal{F}_R(z(t),z(t))+\|z(t)\|^2_{L^2(\Omc)}
	+\|z(t)\|^2_{L^2(\Omega)}
	&=&\dis   \left\langle \zeta(t),z(t)\right\rangle_{(\H_\Omega^s)^\star,\H_\Omega^s} \\
&+&\dis \int_{\Omc} \eta(x,t)z(x,t) dx\\
	&\leq &\dis \frac{1}{2}\|\zeta(t)\|^2_{(\H^s_\Omega)^\star}+\frac{1}{2}\|\eta(t)\|^2_{L^2(\Omc)}\\
	&+&\dis
	\frac{1}{2}\|z(t)\|^2_{L^2(\Omega)}+\dis \frac{1}{2}\|z(t)\|^2_{L^2(\Omc)}.
	\end{array}
$$
	We then can deduce that
	\begin{equation}\label{inter21}
	\dis  \frac 12\frac{d}{dt}\|z(t)\|^2_{L^2(\Omega)}+\frac{1}{2}\|z(t)\|^2_{\H^s_\Omega}\leq
	\dis \frac{1}{2}\|\zeta(t)\|^2_{(\H^s_\Omega)^\star}+\frac{1}{2}\|\eta(t)\|^2_{L^2(\Omc)}.
	\end{equation}
	Integrating \eqref{inter21} over $(0,\tau)$ with $\tau\in [0,T]$ yields
	$$
	\dis \frac{1}{2}\|z(\tau)\|^2_{L^2(\Omega)}+
	\frac{1}{2}\int_0^\tau\|z(t)\|^2_{\H^s_\Omega}dt\leq
	\dis \frac{1}{2}\|\zeta\|^2_{L^2(0,T);(\H^s_\Omega)^\star)}+\frac{1}{2}\|\eta\|^2_{L^2(\Sigma)}+\dis  \frac 12\|\rho_0\|^2_{L^2(\Omega)}.
	$$
	Therefore, we can deduce that
	$$
	\dis \sup_{\tau\in [0,T]}\|z(\tau)\|^2_{L^2(\Omega)}\leq  \|\zeta\|^2_{L^2(0,T);(\H^s_\Omega)^\star)}+\|\eta\|^2_{L^2(\Sigma)}+\dis  \|\rho_0\|^2_{L^2(\Omega)},
	$$
	and
	$$
	\dis \int_0^T\|z(t)\|^2_{\H^s_\Omega}dt\leq \dis \left(\|\zeta\|^2_{L^2((0,T);(\H^s_\Omega)^\star)}+\|\eta\|^2_{L^2(\Sigma)}+\dis  \|\rho_0\|^2_{L^2(\Omega)}\right).
	$$	
Combining \eqref{estimationint3} and \eqref{estimation1baux} we have obtain \eqref{estimation1bauxt}.\par
	\noindent \textbf{Step 4.} We prove uniqueness.\par
	Assume that there exist $z_1\in W(0,T;\H^s_\Omega)$ and $z_2\in W(0,T;\H^s_\Omega)$, solutions to \eqref{modelexist1aux} with the same right hand sides $\zeta,\, \eta$ and initial datum $\rho_0$.  Set $y:=z_1-z_2\in  W(0,T;\H^s_\Omega)$. Then,  $y$ satisfies
	\begin{equation}\label{pazero1}
	\left\{
	\begin{array}{rllll}
	\dis  y_t+(-\Delta)^sy+ y &=&0 \qquad &\mbox{in}& Q,\\
	\mathcal{N}_sy+y&=& 0 &\mbox{in}& \Sigma,\\
	y(.,0)&=& 0 &\mbox{in}& \Omega.
	\end{array}
	\right.
	\end{equation}
	Multiplying the first equation of \eqref{pazero1}  by $y$ and integrating by parts over $Q$ using  \eqref{Int-Part}, it follows that,
	$$\begin{array}{lll}
	0&=&\dis \frac 12\|y(T)\|^2_{L^2(\Omega)}+\dis \int_0^T \mathcal{F}_R(y,y) dt+\|y\|^2_{L^2(Q)}+\|y\|^2_{L^2(\Sigma)}\\
	&\geq&\dis  \frac 12\|y(T)\|^2_{L^2(\Omega)}+\frac{1}{2} \|y\|^2_{L^2((0,T);\H^s_\Omega)}\\
	&\geq&\dis \frac{1}{2}\|y\|^2_{L^2((0,T);\H^s_\Omega)}.
	\end{array}$$
	Hence,  we can deduce that $y=0 $ in $\R^N$. Thus,  $z_1=z_2$ in $\R^N$ and we have shown uniqueness.\\
	The proof is finished.
\end{proof}
The existence result of solutions to the system  \eqref{modeln} is derived from
Theorem \ref{exist1aux} as follows.
\begin{corollary}\label{exist1}
	Let $f\in L^2((0,T);(\H^s_\Omega)^\star),$ $g\in L^2(\Sigma)$,  and $\rho^{0}\in L^2(\Omega)$
	Then, there exists a unique weak solution $\rho\in W(0,T;\H^s_\Omega)$ of \eqref{modeln}.
	Moreover, the following estimates hold:
	\begin{equation}\label{estimation1a}
	\dis \sup_{\tau\in [0,T]}\|\rho(\tau)\|_{L^2(\Omega)}\leq  \dis
	e^T\left( \|f\|_{L^2((0,T);(\H^s_\Omega)^\star)}+\|g\|_{L^2(\mathcal O\times(0,T))}+\dis
	\|\rho_0\|_{L^2(\Omega)}\right),
	\end{equation}	
	\begin{equation}\label{estimation1b}
	\dis \|z\|_{L^2((0,T);\H^s_\Omega)}\leq
	\dis e^T\left(\|f\|_{L^2((0,T);(\H^s_\Omega)^\star)}+\|g\|_{L^2(\Sigma)}+\dis  \|\rho_0\|_{L^2(\Omega)}\right).
	\end{equation}
	and
	\begin{equation}\label{estimation1bt}
	\dis \|z\|_{W(0,T;\H^s_\Omega)}\leq Ce^T\dis
	\left(\|f\|_{L^2((0,T);(\H^s_\Omega)^\star)}+\|g\|_{L^2(\Sigma)}+\dis  \|\rho_0\|_{L^2(\Omega)}\right).
	\end{equation}
	for some constant $C>0$.
\end{corollary}
\begin{proof} Let  $f\in L^2((0,T);(\H^s_\Omega)^\star),$ $g\in L^2(\Sigma)$  and
	$\rho^{0}\in L^2(\Omega)$. Then $\rho=e^{t} z$ is a weak solution of \eqref{modeln}
	if and only if $z$ is a weak solution of \eqref{modelexist1aux}. Consequently, using
	Theorem  \ref{exist1aux}, we have  that there exists a unique $\rho\in W(0,T;\H^s_\Omega)$,
	solution of  \eqref{modeln}.  Moreover, letting $z=e^{-t}\rho$  in from the estimates
	\eqref{estimation1aaux}-\eqref{estimation1bauxt}, we can deduce that
	\eqref{estimation1a}-\eqref{estimation1bt} hold, because for  $t\in [0,T]$, we have
	$$\|e^{-t}\rho\|_{L^2((0,T);\H^s_\Omega)}\geq e^{-T}\|\rho\|_{L^2((0,T);\H^s_\Omega)} \hbox{ and }
	\|e^{-t}\rho(t)\|_{L^2(\Omega)}\geq e^{-T}\|\rho(t)\|_{L^2(\Omega)}.$$
	The proof is finished.
\end{proof}
\begin{proposition}\label{propexistence}
	Let $f\in L^2((0,T);\H^{-s}(\Omega)),$ and $\rho^{0}\in L^2(\Omega)$. Let also $\mathcal{F}(.,.)$ be defined
	as in \eqref{defF}. Then, there exists a unique weak solution $\ro \in W(0,T;\H^s_0(\Omega))$ to \eqref{modeld} in the following sense:
	\begin{equation}\label{Eq-Def31}
	\begin{array}{lll}
	\dis -\int_0^T\langle \phi_t,\rho \rangle_{\H^{-s}(\Omega),\H^s_0(\Omega)}\, dt + \dis
	\int_0^T \mathcal{F}(\ro,\phi)dt
	=\dis \int_0^T \langle f,\, \phi\rangle_{\H^{-s}(\Omega),\H^s_0(\Omega)}\, dt+\int_\Omega \rho_0\,\phi(0) \dx,
	\quad \forall \phi \in H(Q),
	\end{array}
	\end{equation}
	where
	$$ H(Q):=\left\{\xi\in W(0,T;\H^s_0(\Omega)) \hbox{ and } \xi(\cdot, T)=0 \hbox{ a.e.in } \Omega\right\}.$$
	In addition, the following estimates hold
	\begin{equation}\label{estimation2prop}
	\|\ro\|^2_{\C([0,T];L^2(\Omega))}\leq \left[\|\rho_0\|^2_{L^2(\Omega)}+
	\|f\|^2_{L^2((0,T);\H^{-s}(\Omega))}\right],
	\end{equation}
	\begin{equation}\label{estimation1prop}
	\|\ro\|^2_{L^2((0,T);\H^{s}_0(\Omega))}\leq\dis
	\left[\|\rho_0\|^2_{L^2(\Omega)}+\|f\|^2_{L^2((0,T);\H^{-s}(\Omega))}\right]
	\end{equation}
	and
	\begin{equation}\label{estimation1propt}
	\|\ro\|^2_{W((0,T);\H^{s}_0(\Omega))}\leq C\dis
	\left[\|\rho_0\|^2_{L^2(\Omega)}+\|f\|^2_{L^2((0,T);\H^{-s}(\Omega))}\right]
	\end{equation}
	for some constant  $C>0$.	
\end{proposition}
\begin{proof} We proceed  as above for the proof of Theorem \ref{exist1aux}.\par	
	\noindent \textbf{Step 1.} We  first observe that embedding $H(Q)\hookrightarrow L^2((0,T);\H^s_0(\Omega))$ is continuous since 	for any $\rho\in H(Q),$
	$$\|\rho\|_{L^2((0,T);\H^s_0(\Omega)}\leq \|\rho\|_{H(Q)},$$
where  the norm on $L^2((0,T);\H^{s}_0(\Omega))$ is given by
	$$\|\varphi\|^2_{L^2((0,T);\H^s_0(\Omega))}=\int_0^T \|\varphi(\cdot,t)\|^2_{\H^s_0(\Omega)} dt=
	\mathcal{F}(\varphi(\cdot,t),\varphi(\cdot,t))dt,\quad \forall \varphi\in L^2((0,T);\H^{s}_0(\Omega))$$
	and the norm on  $H(Q)$ given by
	$$
	\|\rho\|^2_{H(Q)}:=\|\rho\|^2_{L^2((0,T);\H^s_0(\Omega))}+\|\rho(\cdot,0)\|^2_{L^2\Omega)},\, \forall \rho \in H(Q).
	$$
		We then  consider the bilinear form $\mathcal{E}(\cdot,\cdot)$ defined on
	$L^2((0,T);\H^{s}_0(\Omega))\times H(Q)$ by
	\begin{equation}\label{defCalE1bis}
	\begin{array}{lll}
	\mathcal{E}(\rho,\varphi)&:=&\dis -\int_0^T\langle \varphi_t,\rho \rangle_{(\H^{-s}(\Omega),\H^{s}_0(\Omega)}\, dt
	+\dis\int_0^T \mathcal{F}(\rho,\varphi)dt.
	\end{array}
	\end{equation}	
	Using Cauchy-Schwarz's inequality, the  continuity of the bilinear form $\mathcal F$ and  \eqref{norm0},
	it follows that
	$$\begin{array}{rlll}
	\dis |\mathcal{E}(\rho,\varphi)|
	&\leq& \|\rho\|_{L^2((0,T);\H^{s}(\Omega))}\Big(\|\varphi_{t}\|_{L^2((0,T);\H^{-s}(\Omega))}
	+\|\varphi\|_{L^2((0,T);\H^{s}_0(\Omega))}\Big) .
	\end{array}
	$$
	Thus, for every fixed $\varphi\in H(Q),$
	the functional  $\rho\mapsto \mathcal{E}(\rho,\varphi)$ is continuous on $L^2((0,T);\H^{s}_0(\Omega)).$\par	
	Next,  for every  $\varphi\in H(Q)$,
	$$\begin{array}{rlll}
	\mathcal{E}(\varphi,\varphi)&=&\dis -\int_0^T\langle \varphi_t,\varphi \rangle_{\H^{-s}(\Omega),\H^{s}_0(\Omega)}\, dt	+
	\dis\int_0^T \mathcal{F}(\varphi,\varphi)dt\\
	&\geq&\dis \frac 12  \|\varphi(0)\|^2_{L^2(\Omega)}+\dis\int_0^T \mathcal{F}(\varphi,\varphi)dt
	\\	
	&\geq &
	\dis \dis\frac{1}{2} \|\varphi\|^2_{H(Q)}.
	\end{array}
	$$	
	Finally, we define the functional $L(\cdot):H(Q)\to \R$  by
	$$
	L(\varphi):=\dis \int_0^T \langle f,\, \varphi\rangle_{\H^{-s}(\Omega),\H^s_0(\Omega)}\, dt+
	\int_\Omega \rho_0\,\varphi(0) \dx.$$
	Using Cauchy-Schwarz's inequality and \eqref{norm0}, we get that
	$$\begin{array}{rll}
	|L(\varphi)|
	&\leq &\left(\|\rho_0\|_{L^2(\Omega)}+\|f\|_{L^2((0,T);\H^{-s}(\Omega))}\right)\|\varphi\|_{H(Q)}.
	\end{array}
	$$
	Therefore,  the functional $L(\cdot)$ is continuous on $H(Q)$. Thus, there is $\rho\in L^2((0,T);\H^{s}_0(\Omega))$ such that
	$$
	\mathcal{E}(\rho,\varphi)= L(\varphi),\quad \forall \varphi \in H(Q).
	$$
	Hence,  the system \eqref{modeld} has a solution $\rho\in L^2((0,T);\H^{s}_0(\Omega))$.\\
	\noindent \textbf{Step 2.} We show that $\rho_t\in L^2((0,T);\H^{-s}(\Omega))$\par	
	Since $\rho\in L^2((0,T);\H^{s}_0(\Omega))$,  it follows from Proposition \ref{Prop-22} that
	$(-\Delta)_D^s\rho(\cdot,t)\in \H^{-s}(\Omega)$ and thus
	$\rho_t(\cdot,t)=-(-\Delta)^s_D\rho(\cdot,t) + f(\cdot,t)\in \H^{-s}(\Omega).$
	If we multiply the first equation in \eqref{modeld} by $\phi\in L^2((0,T);\H^{s}_0(\Omega))$,
	and use the integration by parts formula \eqref{Int-Part0},  we obtain
	$$
	\dis
	\left\langle f(t),\phi(t)\right\rangle_{\H^{-s}(\Omega),\H^{s}_0(\Omega)}=	
	\dis  \left\langle \rho_t(t),\phi(t)\right\rangle_{\H^{-s}(\Omega),\H^{s}_0(\Omega)} +
	\dis  \mathcal{F}(\rho(t),\phi(t)).
	$$
	Thus
	$$
	\dis \left|  \left\langle \rho_t(t),\phi(t)\right\rangle_{\H^{-s}(\Omega),\H^{s}_0(\Omega)}\right|
	\leq\|\phi(t)\|_{\H^{s}_0(\Omega)}\left(
	\|f(t)\|_{\H^{-s}(\Omega)}+
	\|\rho(t)\|_{\H^{s}_0(\Omega)}\right),
	$$
	which integrating  over $(0,T)$ gives
	\begin{equation}\label{sam2bis}
	\dis \int_0^T \left|  \left\langle \rho_t(t),\phi(t)\right\rangle_{\H^{-s}(\Omega),\H^{s}_0(\Omega)}\right| dt
	\leq \|\phi\|_{L^2((0,T);\H^{s}_0(\Omega))}
	\left(
	\|f\|_{L^2((0,T);\H^{-s}(\Omega))}+
	\|\rho\|_{L^2((0,T);\H^{s}_0(\Omega)}\right)
	\end{equation}
	Using \eqref{estimation1prop} we get from \eqref{sam2bis} that
	\begin{equation} \label{estimationint3bis}
	\|\rho_t\|_{L^2((0,T);\H^{-s}(\Omega))}\leq C\|\phi\|_{L^2((0,T);\H^{s}_0(\Omega))}
	\left(
	\|f\|_{L^2((0,T);\H^{-s}(\Omega))}+
	\|\rho\|_{L^2((0,T);\H^{s}_0(\Omega)}\right)
	\end{equation}
	for some constant $C>0$. \par
		\noindent \textbf{Step 3.} We show the estimates \eqref{estimation2prop}-\eqref{estimation1propt}.\par
	Multiplying the first equation in \eqref{modeld}  by $\rho\in W((0,T);\H^{s}_0(\Omega))$  and integrating by parts,
	we deduce that
	\begin{equation}\label{inter21bis}
	\dis  \frac 12\frac{d}{dt}\|\rho(t)\|^2_{L^2(\Omega)}+\frac{1}{2}\|\rho(t)\|^2_{\H^{s}_0(\Omega)}\leq
	\dis \frac{1}{2}\|f(t)\|^2_{\H^{s}_0(\Omega)}.
	\end{equation}
Integrating \eqref{inter21bis} over $(0,\tau)$ with $\tau\in [0,T]$ yields
	$$
	\dis \frac{1}{2}\|\rho(\tau)\|^2_{L^2(\Omega)}+
	\frac{1}{2}\int_0^\tau\|\rho(t)\|^2_{\H^{s}_0(\Omega)}dt\leq
	\dis \frac{1}{2}\|f\|^2_{L^2((0,T);\H^{-s}(\Omega))}+\dis  \frac 12\|\rho_0\|^2_{L^2(\Omega)}.
	$$
	Therefore, we can deduce that \eqref{estimation2prop} and \eqref{estimation1prop}.
	Combining \eqref{estimation1prop} and \eqref{estimationint3bis}, we obtain \eqref{estimation1propt}.\par	
	\noindent \textbf{Step 4.} We prove uniqueness.\par
	Assume that there exist $\rho_1\in W(0,T;\H^{s}_0(\Omega))$ and $\rho_2\in W(0,T;\H^{s}_0(\Omega))$, solutions to
	\eqref{modeld} with the same right hand side $f$ and initial datum $\rho_0$.  Set $y:=\rho_1-\rho_2\in
	W(0,T;\H^{s}_0(\Omega))$. Then,  $y$ satisfies
	\begin{equation}\label{pazero2}
	\left\{
	\begin{array}{rllll}
	\dis  y_t+(-\Delta)^sy &=&0 \qquad &\mbox{in}& Q,\\
	y&=& 0 &\mbox{in}& \Sigma,\\
	y(.,0)&=& 0 &\mbox{in}& \Omega.
	\end{array}
	\right.
	\end{equation}
	By multiplying the first equation of \eqref{pazero2}  by $y$ and integrating by parts over $Q$ using
	\eqref{Int-Part0}, it follows that,
	$$
	0	\geq\dis  \frac 12\|y(T)\|^2_{L^2(\Omega)}+\frac{1}{2} \|y\|^2_{L^2((0,T);\H^{s}_0(\Omega))}\\
	\geq\dis \frac{1}{2}\|y\|^2_{L^2((0,T);\H^{s}_0(\Omega))}.
	$$
	Hence,  we can deduce that $y=0 $ in $\R^N$. Thus,  $\rho_1=\rho_2$ in $\R^N$ and we have shown uniqueness.\\
	The proof is finished.
\end{proof}
\section{Weak maximum principle results}
\label{maximum}
In this section, we give some comparison theorems and establish  weak maximum principles for  models \eqref{modeln} and \eqref{modeld}.
We first prove the following result.
	\begin{proposition}	
		Let $\varphi\in \V$, where $\V:=\H_0^{s}(\Omega)$ or $\H^{s}_{\Omega}$. Let also $\mathcal{F}(.,.)$ and $\mathcal{F}_R(.,.)$  be defined as in \eqref{defF} and \eqref{defFR} respectively. If we write $\varphi(x)=\varphi^+(x)-\varphi^-(x)$, where $\varphi^+(x)=\max(\varphi(x),0)$ and $\varphi^-(x)=\max(0,-\varphi(x))$. Then the following inequality holds for $\mathcal{G}=\mathcal{F}$ or $\mathcal{F}_R$ ,
		\begin{equation}\label{ineq}
		-\mathcal{G}(\varphi,\varphi^-)\geq \|\varphi^-\|^2_{\V}\geq 0.
		\end{equation}
	\end{proposition}
	\begin{proof}
		We have
		$$
		\begin{array}{llll}
		\dis 	\mathcal{G}(\varphi,\varphi^-)&= &\dis \frac{C_{N,s}}{2}\dis \int\int_{\left(\RR^{2N}\setminus(\Omc)^2\right)}
		\frac{(\varphi(x)-\varphi(y))(\varphi^-(x)-\varphi^-(y))}{|x-y|^{N+2s}}\;\dx\, \dy\\
		&=&-\dis\frac{C_{N,s}}{2}\dis \int\int_{\left(\RR^{2N}\setminus(\Omc)^2\right)}
		\frac{(\varphi^-(x)-\varphi^-(y))^2}{|x-y|^{N+2s}}\;\dx\, \dy\\
		&&-\dis \frac{C_{N,s}}{2}\dis \int\int_{\left(\RR^{2N}\setminus(\Omc)^2\right)}
		\frac{\varphi^-(x)\rho^+(y)+\rho^+(x)\varphi^-(y)}{|x-y|^{N+2s}}\;\dx\, \dy\\
		&=& \dis-\|\varphi^-\|^2_{\V}-\dis\frac{C_{N,s}}{2} \int\int_{\left(\RR^{2N}\setminus(\Omc)^2\right)}
		\frac{\varphi^-(x)\varphi^+(y)+\varphi^+(x)\varphi^-(y)}{|x-y|^{N+2s}}\;\dx\, \dy.
		\end{array}
		$$
		Using that $\varphi^+(x)\varphi^-(y)\geq0\;\; \hbox{for a.e. }\;\; x,y \in \R^N,$ we obtain \eqref{ineq}.
\end{proof}
\begin{theorem}\label{positive1}
	Let $\rho\in W(0,T;\H^s_0(\Omega))$ be the solution of \eqref{modeld} corresponding to data $f\in L^2(Q)$ and $\rho_0\in L^2(\Omega)$ and $\tilde{\rho}\in W(0,T;\H^s_0(\Omega))$ be the solution of \eqref{modeld} corresponding to data $\tilde{f}\in L^2(Q)$ and $\tilde{\rho}_0\in L^2(\Omega)$. If $\rho_0\leq \tilde{\rho}_0$ almost everywhere in $\Omega$ and $f\leq \tilde{f}$ almost everywhere in $Q$, then $\rho\leq \tilde{\rho}$ almost everywhere in $\R^N\times [0,T]$. In particular, if $f\geq 0$, then $\rho\geq 0$ almost everywhere in $\R^N\times [0,T]$.
\end{theorem}
\begin{proof}
		Let $z=\tilde{\rho}-\rho$. Then, $z\in W(0,T;\H^s_0(\Omega))$ is solution to
		\begin{equation}\label{model10}
		\left\{
		\begin{array}{rllll}
		\dis  z_t+(-\Delta)^sz &=&\tilde{f}-f \qquad &\mbox{in}& Q,\\
		z&=&0  &\mbox{in}& \Sigma,\\
		z(.,0)&=& \tilde{\rho}_0-\rho_0 &\mbox{in}& \Omega.
		\end{array}
		\right.
		\end{equation}
		Let $(x,t)\in \R^N\times [0,T]$, we write $z(x,t)=z^+(x,t)-z^-(x,t)$, where $z^+(x,t)=\max(z(x,t),0)$ and $z^-(x,t)=\max(0,-z(x,t))$. It is sufficient to show that $z^-(x,t)=0$ for almost every $(x,t)\in \R^N\times [0,T]$. Then  we have
		$$\begin{array}{llllll}
		z^-z^+&=& 0 \;\; \hbox{a.e. in } \R^N\times [0,T],\\
		z^-&=&0\;\; \hbox{in } \Sigma,\\
		z^-(x,0)&=&\max(0,-(\tilde{\rho}_0(x)-\rho_0(x))=0\;\; \hbox{a.e. in } \Omega,
		\end{array}$$
		and  $z^-\in W(0,T;\H^s_0(\Omega))$  (see e.g. \cite{War}).
		If we take the duality map between \eqref{model10}  and $\psi\in \H^s_0(\Omega)$, and use the integration by parts \eqref{Int-Part}, we have for any  $t\in [0,T]$,
		$$\begin{array}{llll}
		\dis \langle z_t (t),\psi\rangle_{\H^{-s}(\Omega),\H^s_0(\Omega)}+\mathcal{F}(z(t),\psi)
		=\dis  \dis \int_\Omega (\tilde{f}(t)-f(t))\psi dx.
		\end{array}$$
		Taking $\psi=z^-(t)$ in this latter identity yields
		\begin{equation}\label{ajout1}
		\dis\langle z_t (t),z^-(t)\rangle_{\H^{-s}(\Omega),\H^s_0(\Omega)}+\mathcal{F}(z(t),z^-(t))
		=\dis \int_\Omega (\tilde{f}(t)-f(t))z^-(t)\, \dx,
		\end{equation}
		where $\mathcal{F}$ is defined as in \eqref{defF}.
		Observing  that
		$$
		\begin{array}{llll}
		\dis\langle z_t (t),z^-(t)\rangle_{\H^{-s}(\Omega),\H^s_0(\Omega)}&=&\dis\langle (z^+-z^-)_t (t),z^-(t)\rangle_{\H^{-s}(\Omega),\H^s_0(\Omega)}\\
		&=& -\dis\langle (z^-)_t (t),z^-(t)\rangle_{\H^{-s}(\Omega),\H^s_0(\Omega)}\\
		&=&-\dis\frac{1}{2}\frac{d}{dt}\|z^-(t) \|^2_{L^2(\Omega)}
		\end{array}
		$$
		and using this latter identity, \eqref{ajout1}	becomes,
				$$
		\begin{array}{llll}
		\dis \frac{1}{2}\frac{d}{dt}\|z^-(t) \|^2_{L^2(\Omega)}-\mathcal{F}(z(t),z^-(t))= -\dis \int_\Omega (\tilde{f}(t)-f(t))z^-(t)\, \dx.
		\end{array}
		$$	
		Using \eqref{ineq} and the fact that $\tilde{f}-f\geq 0$ we deduce that
		\begin{equation}
		\dis \frac{1}{2}\frac{d}{dt}\|z^-(t) \|^2_{L^2(\Omega)}\leq 0
		\end{equation}
		and it follows from  the Gronwall's Lemma  that
		$$\|z^-(t) \|^2_{L^2(\Omega)}\leq \|z^-(\cdot,0)\|_{L^2(\Omega)}=0,$$
		because $z^-(x,0)=0$ a.e. in $\Omega$.
		Hence
		$z^-(x,t)=0$ for almost every $(x,t)\in \Omega\times [0,T]$ and since $z^-=0$ in $\Sigma$, it follows that $z^-=0$  almost everywhere  in $\R^N\times [0,T]$. Consequently,  $\tilde{\rho}\geq\rho $ almost everywhere  in $\R^N\times [0,T]$.
\end{proof}
\begin{theorem}\label{theominmax1} Let $f\in L^\infty(Q)$ be such that $f\leq 0$ a.e in $Q$ and $\rho_0\in L^\infty (\Omega)$. Then the weak solution of \eqref{modeld} satisfies
		\begin{equation}\label{minmax2}
		\rho(x,t)\leq \dis \|f\|_{L^\infty(Q)}+ \|\rho_0\|_{L^\infty(\Omega)}\hbox{ a.e. in } \R^N \times [0,T].
		\end{equation}
	\end{theorem}
	\begin{proof}		
		We set $w(x,t)=\|f\|_{L^\infty(Q)}+ \|\rho_0\|_{L^\infty(\Omega)}-\rho(x,t)$ for any $(t,x)\in \R^N \times (0,T)$. Then,
		$$
		w(x,0)=\|f\|_{L^\infty(Q)}+\|\rho^0\|_{L^\infty(\Omega)}-\rho^0(x)\geq 0,\quad \forall x\in \Omega.
		$$
		 Moreover, $w$ satisfies
		\begin{equation}\label{modelwb}
		\left\{
		\begin{array}{rllll}
		\dis  w_t+(-\Delta)^sw &=& -f\qquad &\mbox{in}& Q,\\
		w&=&\|f\|_{L^\infty(Q)}+\|\rho^0\|_{L^\infty(\Omega)} &\mbox{in}& \Sigma,\\
		w(.,0)&=&\|f\|_{L^\infty(Q)}+\|\rho^0\|_{L^\infty(\Omega)}-\rho^0 &\mbox{in}& \Omega.
		\end{array}
		\right.
		\end{equation}
		We write for $(x,t)\in \R^N\times [0,T]$, $w(x,t)=w^+(x,t)-w^-(x,t)$, where $w^+(x,t)=\max(w(x,t),0)$ and $w^-(x,t)=\max(0,-w(x,t))$. It is sufficient to show that $w^-(x,t)=0$ for almost every $(x,t)\in \R^N\times [0,T]$. Since $\|f\|_{L^\infty(Q)}+\|\rho^0\|_{L^\infty(\Omega)}\geq 0$, we have that   $w^-= 0 \hbox{ on } \Sigma$. We also have $w^-\in W(0,T;\H^s_0(\Omega))$.
		If we take the duality map between \eqref{modelwb}  and $w^-\in W(0,T;\H^s_0(\Omega))$, and use the integration by parts \eqref{Int-Part}, we have for any  $t\in [0,T]$,
				\begin{equation}\label{ajout10}
		\dis -\frac{1}{2}\frac{d}{dt}\|w^-(t) \|^2_{L^2(\Omega)}+\mathcal{F}(w(t),w^-(t))
		=-\dis \int_\Omega f(t)w^-(t)\, \dx,
		\end{equation}
		where $\mathcal{F}$ is defined as in \eqref{defF}.
		Hence,
		$$\dis \frac{1}{2}\frac{d}{dt}\|w^-(t) \|^2_{L^2(\Omega)}-\mathcal{F}(w(t),w^-(t))
		=\dis \int_\Omega f(t)w^-(t)\, \dx.
		$$
		Since $\dis f(t,x)\leq 0$ and $w^-(t,x)\geq 0$ for a.e $(x,t)\in \R^N\times[0,T]$, using  \eqref{ineq}, we deduce that
		\begin{equation}
		\dis \frac{1}{2}\frac{d}{dt}\|w^-(t) \|^2_{L^2(\Omega)}\leq 0.
		\end{equation}
		By using the Gronwall's Lemma it follows that
		$$\|w^-(t) \|^2_{L^2(\Omega)}\leq \|w^-(\cdot,0)\|_{L^2(\Omega)}=0,$$
		because $w^-(x,0)=0$ a.e. in $\Omega$. Hence
		$w^-(x,t)=0$ for almost every $(x,t)\in \Omega\times [0,T]$ and since $w^-=0$ in $\Sigma$, it follows that $w^-=0$  almost everywhere  in $\R^N\times [0,T]$. Consequently,  $w\geq 0$ almost everywhere  in $\R^N\times [0,T]$ and so the claim is proved. This completes the proof.
\end{proof}
\begin{remark}
	Note that from Theorem \ref{theominmax1}, we deduce that
	\begin{equation}\label{max0}
	\|\rho\|_{L^\infty(\R^N\times[0,T])}\leq  \|f\|_{L^\infty(Q)}+\|\rho_0\|_{L^\infty(\Omega)}.
	\end{equation}
\end{remark}
In order to prove the weak maximum principle for model \eqref{modeln}, we first prove the following results for the auxiliary model \eqref{modelexist1aux}.
\begin{theorem}\label{positive2}
		Let  $z\in W(0,T;\H^s_\Omega)$ be the solution of \eqref{modelexist1aux} associated to data $\zeta\in L^2(Q)$, $\eta\in L^2(\Sigma)$ and $\rho_0\in L^2(\Omega)$ and $\tilde{z}\in W(0,T;\H^s_\Omega)$ be the solution of \eqref{modelexist1aux} associated to data $\tilde{\zeta}\in L^2(Q)$, $\tilde{\eta}\in L^2(\Sigma)$ and $\tilde{\rho}_0\in L^2(\Omega)$. If $\rho_0\leq \tilde{\rho}_0$ almost everywhere in $\Omega$, $\eta\leq \tilde{\eta}$ almost everywhere in $\Sigma$ and $\zeta\leq \tilde{\zeta}$ almost everywhere in $Q$, then $z\leq \tilde{z}$ almost everywhere in $\R^N\times [0,T]$. In particular, if $\eta\geq 0$ and $\rho_{0}\geq 0$, then $z\geq 0$ almost everywhere in $\R^N\times [0,T]$.
	\end{theorem}
	\begin{proof}
		Let $u=\tilde{z}-z$. Then, $u\in W(0,T;\H^s_\Omega)$ is solution to
		\begin{equation}\label{model11}
		\left\{
		\begin{array}{rllll}
		\dis  u_t+(-\Delta)^su+u &=&\tilde{\zeta}-\zeta \qquad &\mbox{in}& Q,\\
		\mathcal{N}_su+u&=& \tilde{\eta}-\eta  &\mbox{in}& \Sigma,\\
		u(.,0)&=& \tilde{\rho}_0-\rho_0 &\mbox{in}& \Omega.
		\end{array}
		\right.
		\end{equation}
		Let $(x,t)\in\R^N\times [0,T]$, we write $u(x,t)=u^+(x,t)-u^-(x,t)$, where $u^+(x,t)=\max(u(x,t),0)$ and $u^-(x,t)=\max(0,-u(x,t))$. It is sufficient to show that $u^-(x,t)=0$ for almost every $(x,t)\in \R^N\times [0,T]$. Then  we have that
		$$\begin{array}{llllll}
		u^-u^+&=& 0 \;\; \hbox{a.e. in } \R^N\times [0,T],\\
		u^-(x,0)&=&\max(0,-(\tilde{\rho}_0(x)-\rho_0(x))=0\;\; \hbox{a.e. in } \Omega\\
		\end{array}$$
		and  $u^-\in W(0,T;\H^s_\Omega).$
		If we take the duality map between \eqref{model10}  and $\psi\in \H^s_{\Omega}$, and use the integration by parts \eqref{Int-Part}, we have for any  $t\in [0,T]$,
		$$\begin{array}{llll}
		\dis \langle u_t (t),\psi\rangle_{(\H^s_{\Omega})^*,\H^s_{\Omega}}+\mathcal{F}_R(u(t),\psi)+\int_{\Omega}u(t)\psi\; \dx+\int_{\Omc}u(t)\psi \;\dx\\
		=\dis \int_\Omega (\tilde{\zeta}(t)-\zeta(t))\psi\; \dx+\int_{\Omc} (\tilde{\eta}(t)-\eta(t))\psi\; \dx,
		\end{array}$$
		where $\mathcal{F}_R$ is defined as in \eqref{defFR}.
		Taking $\psi=u^-(t)$ in this latter identity yields
		$$
		\begin{array}{llll}
		\dis -\frac{1}{2}\frac{d}{dt}\|u^-(t) \|^2_{L^2(\Omega)}+\mathcal{F}_R(u(t),u^-(t))
		\dis -\int_{\Omega}(u^-(t))^2\;\dx-\int_{\Omc}(u^-(t))^2\;\dx\\
		\qquad	= \dis \int_\Omega (\tilde{\zeta}(t)-\zeta(t))u^-(t)\; \dx+\int_{\Omc} (\tilde{\eta}(t)-\eta(t))u^-(t)\; \dx.
		\end{array}
		$$
				Using \eqref{ineq} and the fact that $\tilde{\zeta}-\zeta\geq 0$ almost everywhere in $Q$ and $\tilde{\eta}-\eta\geq 0$ almost everywhere in $\Sigma$, we deduce that
		\begin{equation}
		\begin{array}{llll}
		\dis \frac{1}{2}\frac{d}{dt}\|u^-(t) \|^2_{L^2(\Omega)}\leq 0.
		\end{array}
		\end{equation}
		By using the Gronwall's Lemma it follows that
		$$\|u^-(t) \|^2_{L^2(\Omega)}\leq \|u^-(\cdot,0)\|_{L^2(\Omega)}=0,$$
		because $u^-(x,0)=0$ a.e. in $\Omega$.
		Hence
		$u^-(x,t)=0$ for almost every $(x,t)\in \Omega\times [0,T]$ and since $u^-=0$ in $\Sigma$, it follows that $u^-=0$
		almost everywhere  in $\R^N\times [0,T]$. Consequently,  $\tilde{z}\geq z $ almost everywhere  in $\R^N\times [0,T]$.	
\end{proof}
\begin{corollary}\label{positive22}
	Let $\rho\in W(0,T;\H^s_\Omega)$ be the solution of \eqref{modeln} associated to data $f\in L^2(\Omega)$, $g\in L^2(\Sigma)$ and $\rho_0\in L^2(\Omega)$ and $\tilde{\rho}\in W(0,T;\H^s_\Omega)$ be the solution of \eqref{modeln} associated to data $\tilde{f}\in L^2(Q)$, $\tilde{g}\in L^2(\Sigma)$ and $\tilde{\rho}_0\in L^2(\Omega)$. If $\rho_0\leq \tilde{\rho}_0$ almost everywhere in $\Omega$, $g\leq \tilde{g}$ almost everywhere in $\Sigma$ and $f\leq \tilde{f}$ almost everywhere in $Q$, then $\rho\leq \tilde{\rho}$ almost everywhere in $\R^N\times [0,T]$. In particular, if $g\geq 0$, $f\geq 0$ and $\rho_{0}\geq 0$, then $\rho\geq 0$ almost everywhere in $\R^N\times [0,T]$.
\end{corollary}
\begin{proof}
	Since, $\rho_0\leq \tilde{\rho}_0$ almost everywhere in $\Omega$, $g\leq \tilde{g}$ almost everywhere in $\Sigma$ and $f\leq \tilde{f}$ almost everywhere in $Q$, then using \eqref{inter1}, we have that  $\eta\leq \tilde{\eta}$ almost everywhere in $\Sigma$ and $\zeta\leq \tilde{\zeta}$ almost everywhere in $Q$. Thus, $z=e^{-t}\rho$ is the solution of \eqref{modelexist1aux} associated to data $\zeta$, $\eta$ and $\rho_0$ and $\tilde{z}=e^{-t}\tilde{\rho}$ is the solution of \eqref{modelexist1aux} associated to data $\tilde{\zeta}$, $\tilde{\eta}$ and $\tilde{\rho}_0$. Hence, using   Theorem \ref{positive2}, we deduce  that $z\leq \tilde{z}$ a.e in $\R^N\times [0,T]$. This means that $e^{-t}\rho\leq e^{-t}\tilde{\rho}$ a.e in $\R^N\times [0,T]$. Therefore, $\rho\leq \tilde{\rho}$ a.e in $\R^N\times [0,T]$. This completes the proof.
\end{proof}
\begin{theorem}\label{theominmax12} Let
		$\zeta\in L^\infty(Q),$ $\eta\in L^\infty(\Sigma)$ and $\rho_0\in L^\infty(\Omega)$. Then  we have the following boundedness result for the weak solution of \eqref{modelexist1aux}:
		\begin{equation}\label{minmax12}
		z(x,t)\leq\|\rho_0\|_{L^\infty(\Omega)}+\|\zeta\|_{L^\infty(Q)}+\|\eta\|_{L^\infty(\Sigma)} \hbox{ a.e. in } \R^N \times [0,T].
		\end{equation}
	\end{theorem}
	\begin{proof}
		Set
		$K=\dis \|\rho_0\|_{L^\infty(\Omega)}+\|\zeta\|_{L^\infty(Q)}+\|\eta\|_{L^\infty(\Sigma)}$ and $w(x,t)=K-z(x,t)$ for any $(x,t)\in \R^N \times [0,T]$. Then
		\begin{subequations}\label{eq}
			\begin{alignat}{11}
			K\geq&\; 0,\label{eqajout1}\\
			K\geq &\;\zeta(x,t) \;\;\;\text{for a.e}\; (x,t)\in Q,\label{eqajout2}\\
			K\geq &\;\eta(x,t) \;\;\;\text{for a.e}\; (x,t)\in \Sigma,\label{eqajout3}\\
			K\geq&\;\rho_0(x) \;\;\;\text{for a.e} \; \Omega.\label{eqajout4}
			\end{alignat}
		\end{subequations}
		In addition,
		\begin{equation}\label{eqajout5}
		w(x,0)=K-\rho_0(x)\geq 0 \;\hbox{   for  a.e. }x\in \Omega.
		\end{equation}
		Moreover, $w$ satisfies
		\begin{equation}\label{modelwb1}
		\left\{
		\begin{array}{rllll}
		\dis  w_t+(-\Delta)^sw + w&=& K-\zeta \qquad &\mbox{in}& Q,\\
		\mathcal{N}_sw+w&=&K-\eta  &\mbox{in}& \Sigma,\\
		w(.,0)&=& K-\rho_0 &\mbox{in}& \Omega.
		\end{array}
		\right.
		\end{equation}
		For $(x,t)\in \R^N\times [0,T]$, we write  $w(x,t)=w^+(x,t)-w^-(x,t)$, where $w^+(x,t)=\max(w(x,t),0)$ and $w^-(x,t)=\max(0,-w(x,t))$. It is sufficient to show that $w^-(x,t)=0$ for almost every $(x,t)\in \R^N\times [0,T]$. Then  we have that
		$$\begin{array}{llllll}
		w^-w^+&=& 0 \;\; \hbox{a.e. in } \R^N\times [0,T],\\
		w^-(x,0)&=&\max(0,-w(x,0))=0\;\; \hbox{a.e. in } \Omega,\\
		\end{array}$$
		and  $w^-\in W(0,T;\H^s_\Omega).$
		If we take the duality map between \eqref{modelwb1}  and $\psi\in \H^s_{\Omega}$, and use the integration by parts \eqref{Int-Part}, we have for any  $t\in [0,T]$,
		$$
		\begin{array}{llll}
		\dis \dis \langle w_t (t),\psi\rangle_{(\H^s_{\Omega})^*,\H^s_{\Omega}}+\mathcal{F}_R(w(t),\psi)+\int_{\Omega}w(t)\psi \;\dx\\
		=\dis \int_\Omega (K-\zeta(t))\psi \;\dx+\int_{\Omc} (K-\eta(t))\psi \;\dx,
		\end{array}
		$$
		where $\mathcal{F}_R$ is defined as in \eqref{defFR}. Taking $\psi=w^-(t)$, this latter identity gives
		$$
		\begin{array}{llll}
		\dis \dis \langle w_t (t),w^-(t)\rangle_{(\H^s_{\Omega})^*,\H^s_{\Omega}}+\mathcal{F}_R(w(t),w^-(t))+\int_{\Omega}w(t)w^-(t) \;\dx+\int_{\Omc}w(t)w^-(t)\;\dx\\
		\qquad	=\dis \int_\Omega (K-\zeta(t))w^-(t) \;\dx+\int_{\Omc} (K-\eta(t))w^-(t) \;\dx.
		\end{array}
		$$
		Hence,
		\begin{equation}\label{ajout121}
\begin{array}{llll}
		\dis \frac{1}{2}\frac{d}{dt}\|w^-(t) \|^2_{L^2(\Omega)} -\mathcal{F}_R(w(t),w^-(t))-\int_{\Omega}(w^-(t))^2dx-\int_{\Omc}(w^-(t))^2\;\dx\\
		\quad=-\dis \int_\Omega (K-\zeta(t))w^-(t) \;\dx-\int_{\Omc} (K-\eta(t))w^-(t) \;\dx.
		\end{array}
		\end{equation}	
 Next, using \eqref{eqajout2}, \eqref{eqajout3} and  \eqref{ineq}, we deduce from \eqref{ajout121} that
		$$\begin{array}{llll}
		\dis \frac{1}{2}\frac{d}{dt}\|w^-(t) \|^2_{L^2(\Omega)}
		\leq 0.
		\end{array}
		$$
		By using the Gronwall's Lemma it follows that
		$$\|w^-(t) \|^2_{L^2(\Omega)}\leq \|w^-(\cdot,0)\|_{L^2(\Omega)}=0,$$
		because $w^-(x,0)=0$ a.e. in $\Omega$.
		Hence
		$w^-(x,t)=0$ for almost every $(x,t)\in \Omega\times [0,T]$ and since $w^-=0$ in $\Sigma$, it follows that $w^-=0$  almost everywhere  in $\R^N\times [0,T]$. Consequently,  $z\leq K $ almost everywhere  in $\R^N\times [0,T]$ and \eqref{minmax12} holds.	
\end{proof}
\begin{corollary}\label{coro2}
		Let  $f\in L^\infty(Q),$  $g\in L^\infty(\Sigma)$ and $\rho_0\in L^\infty(\Omega)$. Then  we have the following boundedness result for the weak solution of \eqref{modeln}:
		\begin{equation}\label{minmax21}
		\rho(x,t)\leq e^{T}\left(\|\rho_0\|_{L^\infty(\Omega)}+\|f\|_{L^\infty(Q)}+\|g\|_{L^\infty(\Sigma)}\right) \hbox{ a.e. in } \R^N \times [0,T].
		\end{equation}
	\end{corollary}
	\begin{proof}
		Let $f\in L^\infty(Q),$ $g\in L^\infty(\Sigma)$ and $\rho$ be the weak solution of \eqref{modeln}. Then $\zeta=e^{-t}f\in L^\infty(Q)$ and $\eta=e^{-t}g\in L^\infty(\Sigma)$.
		Thus, $z=e^{-t}\rho$ is the solution of \eqref{modelexist1aux} associated to data $\zeta$, $\eta$ and $\rho_0$. Therefore, using   Theorem \ref{theominmax12},  we obtain
		\begin{equation*}
		e^{-t}\rho(x,t)\leq \left(\|\rho_0\|_{L^\infty(\Omega)}+\|\zeta\|_{L^\infty(Q)}+\|\eta\|_{L^\infty(\Sigma)}\right) \hbox{ a.e. in } \R^N \times [0,T].
		\end{equation*}
		Hence,
		\begin{equation*}
		e^{-t}\rho(x,t)\leq \left(\|\rho_0\|_{L^\infty(\Omega)}+\|f\|_{L^\infty(Q)}+\|g\|_{L^\infty(\Sigma)}\right) \hbox{ a.e. in } \R^N \times [0,T].
		\end{equation*}
		Therefore,
		\begin{equation}\label{minmax121}
		\rho(x,t)\leq e^{T}\left(\|\rho_0\|_{L^\infty(\Omega)}+\|f\|_{L^\infty(Q)}+\|g\|_{L^\infty(\Sigma)}\right) \hbox{ a.e. in } \R^N \times [0,T].
		\end{equation}
		The proof is finished.
	\end{proof}
	\begin{remark}
		Note that from Corollary \ref{coro2}, we deduce the following
		\begin{equation}\label{max}
		\|\rho\|_{L^\infty(\R^N\times[0,T])}\leq e^{T}\left[\dis \|\rho_0\|_{L^\infty(\Omega)}+\|f\|_{L^\infty(Q)}+\|g\|_{L^\infty(\Sigma)}\right].
		\end{equation}
\end{remark}

\begin{remark}
		We notice that solutions of the systems \eqref{modeln}-\eqref{modeld} can be represented by using the semigroups theory as introduced in \cite{claus2020realization,GW-CPDE}. This could simplify the proofs of the manuscript.
\end{remark}
\section{Conclusion}
\label{conclusion}
In this article, we established  existence results for a class of linear parabolic equations involving the fractional Laplace operator of order $s$ ($0<s<1$). We proved the positiveness  of the solutions under the assumption that the data are positive and established  weak maximum principle for these equations. We also proved that the  solutions of the equations considered in the paper are  bounded in $L^\infty(\R^N\times[0,T])$ if the initial condition is bounded in $L^\infty(\Omega)$ and the source term is bounded in $L^\infty(Q)$.
We believe that the boundedness results obtained in this paper  can, for instance,  contribute significantly to prove the well posedness of some nonlinear systems  involving fractional Laplace operators as well as  the  uniqueness results of optimal control  associated to such  systems.

\bibliographystyle{abbrv}
\bibliography{mybibfile}

\begin{thebibliography}{10}

\bibitem{al2014}
M.~Al-Refai and Y.~Luchko.
\newblock Maximum principle for the fractional diffusion equations with the
  riemann-liouville fractional derivative and its applications.
\newblock {\em Fractional Calculus and Applied Analysis}, 17(2):483--498, 2014.

\bibitem{ATW}
W.~Arendt, A.~F.~M. ter Elst, and M.~Warma.
\newblock {F}ractional powers of sectorial operators via the
  {D}irichlet-to-{N}eumann operator.
\newblock {\em Comm. Partial Differential Equations}, 43(1):1--24, 2018.

\bibitem{barrios2012}
B.~Barrios, E.~Colorado, A.~De~Pablo, and U.~S{\'a}nchez.
\newblock On some critical problems for the fractional laplacian operator.
\newblock {\em Journal of Differential Equations}, 252(11):6133--6162, 2012.

\bibitem{brandle2013}
C.~Br{\"a}ndle, E.~Colorado, A.~de~Pablo, and U.~S{\'a}nchez.
\newblock A concave—convex elliptic problem involving the fractional
  laplacian.
\newblock {\em Proceedings of the Royal Society of Edinburgh Section A:
  Mathematics}, 143(1):39--71, 2013.

\bibitem{cabre2014}
X.~Cabr{\'e} and E.~Cinti.
\newblock Sharp energy estimates for nonlinear fractional diffusion equations.
\newblock {\em Calculus of Variations and Partial Differential Equations},
  49(1):233--269, 2014.

\bibitem{caffarelli2008}
L.~A. Caffarelli, S.~Salsa, and L.~Silvestre.
\newblock Regularity estimates for the solution and the free boundary of the
  obstacle problem for the fractional laplacian.
\newblock {\em Inventiones mathematicae}, 171(2):425--461, 2008.

\bibitem{caffarelli2010}
L.~A. Caffarelli and A.~Vasseur.
\newblock Drift diffusion equations with fractional diffusion and the
  quasi-geostrophic equation.
\newblock {\em Annals of Mathematics}, 171(3):1903--1930, 2010.

\bibitem{chechkin2005}
A.~V. Chechkin, R.~Gorenflo, and I.~M. Sokolov.
\newblock Fractional diffusion in inhomogeneous media.
\newblock {\em Journal of Physics A: Mathematical and General}, 38(42):L679,
  2005.

\bibitem{claus2020realization}
B.~Claus and M.~Warma.
\newblock Realization of the fractional laplacian with nonlocal exterior
  conditions via forms method.
\newblock {\em Journal of Evolution Equations}, 20(4):1597--1631, 2020.

\bibitem{NPV}
E.~Di~Nezza, G.~Palatucci, and E.~Valdinoci.
\newblock Hitchhiker's guide to the fractional {S}obolev spaces.
\newblock {\em Bull. Sci. Math.}, 136(5):289--307, 2012.

\bibitem{Dipierro2017}
S.~Dipierro, X.~Ros-Oton, and E.~Valdinoci.
\newblock Nonlocal problems with {N}eumann boundary conditions.
\newblock {\em Rev. Mat. Iberoam}, 33(2):377--416, 2017.

\bibitem{freed2002}
A.~Freed, K.~Diethelm, and Y.~Luchko.
\newblock Fractional-order viscoelasticity (fov): Constitutive development
  using the fractional calculus: First annual report.
\newblock {\em First annual report (No. NAS 1.15: 211914)}, 2002.

\bibitem{GW-CPDE}
C.~G. Gal and M.~Warma.
\newblock Nonlocal transmission problems with fractional diffusion and boundary
  conditions on non-smooth interfaces.
\newblock {\em Comm. Partial Differential Equations1}, 42(4):579--625, 2017.

\bibitem{hilfer2000}
R.~Hilfer.
\newblock {\em Applications of fractional calculus in physics}.
\newblock World scientific, 2000.

\bibitem{kilbas2006}
A.~A. Kilbas, H.~M. Srivastava, and J.~J. Trujillo.
\newblock {\em Theory and applications of fractional differential equations},
  volume 204.
\newblock elsevier, 2006.

\bibitem{lions1971}
J.~Lions.
\newblock {\em Optimal control of systems governed partial differential
  equations}.
\newblock Springer, NY, 1971.

\bibitem{lions2013}
J.~L. Lions.
\newblock {\em Equations diff{\'e}rentielles op{\'e}rationnelles: et
  probl{\`e}mes aux limites}, volume 111.
\newblock Springer-Verlag, 2013.

\bibitem{luchko2009}
Y.~Luchko.
\newblock Maximum principle for the generalized time-fractional diffusion
  equation.
\newblock {\em Journal of Mathematical Analysis and Applications},
  351(1):218--223, 2009.

\bibitem{luchk2010}
Y.~Luchko.
\newblock Some uniqueness and existence results for the initial-boundary-value
  problems for the generalized time-fractional diffusion equation.
\newblock {\em Computers \& Mathematics with Applications}, 59(5):1766--1772,
  2010.

\bibitem{metzler2000}
R.~Metzler and J.~Klafter.
\newblock The random walk's guide to anomalous diffusion: a fractional dynamics
  approach.
\newblock {\em Physics reports}, 339(1):1--77, 2000.

\bibitem{milchev2009}
A.~Milchev, J.~L. Dubbeldam, V.~G. Rostiashvili, and T.~A. Vilgis.
\newblock Polymer translocation through a nanopore: A showcase of anomalous
  diffusion.
\newblock {\em Annals of the New York Academy of Sciences}, 1161(1):95--104,
  2009.

\bibitem{musina2019}
R.~Musina and A.~I. Nazarov.
\newblock Strong maximum principles for fractional laplacians.
\newblock {\em Proceedings of the Royal Society of Edinburgh Section A:
  Mathematics}, 149(5):1223--1240, 2019.

\bibitem{podlubny1998}
I.~Podlubny.
\newblock {\em Fractional differential equations: an introduction to fractional
  derivatives, fractional differential equations, to methods of their solution
  and some of their applications}.
\newblock Elsevier, 1998.

\bibitem{servadei}
R.~Servadei and E.~Valdinoci.
\newblock Mountain {P}ass solutions for non-local elliptic operators.
\newblock {\em J. Math. Anal. Appl.}, 389:887--898, 2012.

\bibitem{silvestre2007}
L.~Silvestre.
\newblock Regularity of the obstacle problem for a fractional power of the
  laplace operator.
\newblock {\em Communications on Pure and Applied Mathematics: A Journal Issued
  by the Courant Institute of Mathematical Sciences}, 60(1):67--112, 2007.

\bibitem{trong2021}
N.~N. Trong, B.~L.~T. Thanh, et~al.
\newblock On the strong maximum principle for a fractional laplacian.
\newblock {\em Archiv der Mathematik}, pages 1--11, 2021.

\bibitem{viet2019}
T.~Q. Viet, N.~M. Dien, and D.~D. Trong.
\newblock Stability of solutions of a class of nonlinear fractional laplacian
  parabolic problems.
\newblock {\em Journal of Computational and Applied Mathematics}, 355:51--76,
  2019.

\bibitem{wang2020}
G.~Wang, X.~Ren, and D.~Baleanu.
\newblock Maximum principle for hadamard fractional differential equations
  involving fractional laplace operator.
\newblock {\em Mathematical Methods in the Applied Sciences}, 43(5):2646--2655,
  2020.

\bibitem{War}
M.~Warma.
\newblock The fractional relative capacity and the fractional {L}aplacian with
  {N}eumann and {R}obin boundary conditions on open sets.
\newblock {\em Potential Anal.}, 42(2):499--547, 2015.

\bibitem{War-ACE}
M.~Warma.
\newblock Approximate controllability from the exterior of space-time
  fractional diffusive equations.
\newblock {\em SIAM J. Control Optim.}, 57(3):2037--2063, 2019.

\end{thebibliography}
\end{document}